\documentclass
{amsart}
\usepackage{amssymb,amsmath,amsthm,amsfonts,amsopn,url,color,hyperref,enumerate,mathtools,microtype,MnSymbol}
\usepackage[normalem]{ulem}
\usepackage[all]{xy}
\input xy
\xyoption{all}
\usepackage{amscd}
\usepackage{soul}

\theoremstyle{plain}
\newtheorem{thm}{Theorem}

\newtheorem{lemma}[thm]{Lemma}

\theoremstyle{definition}

\newtheorem*{defn*}{Definition}
\newtheorem*{question*}{Question}

\newtheorem{example}[thm]{Example}
\newtheorem*{example*}{Example}

\newtheorem*{rem*}{Remark}

\newcommand{\field}[1]{\mathbb{#1}}
\newcommand{\N}{\field{N}}

\newcommand{\ideal}[1]{\mathfrak{#1}}
\newcommand{\m}{\ideal{m}}

\newcommand{\p}{\ideal{p}}

\newcommand{\func}[1]{\mathrm{#1} \,}

\newcommand{\Spec}{\func{Spec}}

\newcommand{\ra}{\rightarrow}

\DeclareMathOperator{\ann}{ann}

\newcommand{\be}{\begin{enumerate}}
\newcommand{\ee}{\end{enumerate}}

\newcommand{\li}
 {\leftfootline}

\newcommand{\cI}{\mathcal{I}}

\renewcommand{\phi}{\varphi}

\let\int\relax
\DeclareMathOperator{\int}{i}

\author{Neil Epstein}
\address{Department of Mathematical Sciences \\ George Mason University \\ Fairfax, VA  22030}
\email{nepstei2@gmu.edu}

\title{The McCoy property in Ohm-Rush algebras}

\date{October 8, 2020}
\begin{document}
\maketitle

\begin{abstract}
  An Ohm-Rush algebra $R \ra S$ is called \emph{McCoy} if for any zero-divisor $f$ in $S$, its content $c(f)$ has nonzero annihilator in $R$, because McCoy proved this when $S=R[x]$.  We answer a question of Nasehpour by giving an example of a   faithfully flat Ohm-Rush algebra with the McCoy property that is not a weak content algebra.  However, we show that a faithfully flat Ohm-Rush algebra is a weak content algebra iff $R/I \ra S/I S$ is McCoy for all  radical (resp. prime) ideals $I$ of $R$.  When $R$ is Noetherian (or has the more general \emph{fidel (A)} property), we show that it is equivalent that $R/I \ra S/IS$ is McCoy for all ideals.  
\end{abstract}

Nearly 70 years ago, McCoy \cite{Mcc-divzero} showed that if a polynomial with coefficients in a given commutative ring is a zero-divisor, then there is some nonzero element of the base ring that annihilates all of the coefficients.  Fields \cite[Theorem 5]{Fi-zdnilps} proved the same thing when $R$ is Noetherian and $S=R[\![x]\!]$.  Nashepour \cite[Theorem 2]{Nas-zdsgmod} proved the same thing when $S=R[M]$, provided that the commutative monoid $M$ is torsion-free and cancellative.  Interestingly, in this case he shows that these properties for $M$ are also equivalent to $R \ra S$ being a content algebra.  Extending this result \cite[Theorem 3]{Nas-zdcontent}, he shows that it is also equivalent that $R\ra S$ is a weak content algebra, or even the property that for any pair of ``polynomials’’ with unit content, their product also has unit content.

Let $R \ra S$ be a ring homomorphism.  Recall that for $f\in S$, the \emph{content} of $f$, denoted $c(f)$, is the intersection of all those ideals $I$ of $R$ for which $f\in IS$.  If $c(f)$ is itself such an ideal (i.e. there is a minimum such ideal), we call the algebra \emph{Ohm-Rush}.  We call an Ohm-Rush algebra a \emph{content algebra} if it is faithfully flat and satisfies the Dedekind-Mertens property, namely that for any $f,g\in S$, there is some $n\in \N$ with $c(f)^{n+1} c(g) = c(f)^n c(fg)$.  See \cite{OhmRu-content}.  For instance, $R \ra R[x]$ is a content algebra \cite{Ded-DM, Mer-DM}, and if $R$ is Noetherian, then $R \ra R[\![x]\!]$ is a content algebra \cite{nmeSh-DMpower}.  An intermediate property is that of a \emph{weak content algebra} \cite{Ru-content}, which is an Ohm-Rush algebra such that $c(fg)$ and $c(f)c(g)$ always have the same radical.  It is equivalent to posit that for any prime ideal $\p \in \Spec R$, either $\p S = S$ or $\p S$ is prime.  Between weak content algebra and content algebra lies the notion of a \emph{semicontent algebra} \cite{nmeSh-OR}, which is a faithfully flat Ohm-Rush algebra such that for any multiplicative subset $W \subseteq R$ and any $f,g\in S$ with $c(f) \cap W \neq \emptyset$, we have $c(fg)_W = c(g)_W$.  To date nobody knows whether there is a semicontent algebra that fails to be a content algebra, nor whether there is a faithfully flat weak content algebra that fails to be a semicontent algebra.  However, the author and Shapiro showed \cite{nmeSh-OR2} that if $R$ is Noetherian, every faithfully flat weak content algebra over $R$ must be a semicontent algebra.
    
Recall that an Ohm-Rush $R$-algebra $R \ra S$ is called \emph{McCoy} \cite[Definition 2.9]{Nas-Aus} if whenever $f,g\in S$ with $fg=0$ and $g\neq 0$, there is some nonzero $r\in R$ with $rc(f)=0$. In that article, Nasehpour notes (citing in turn \cite[6.1]{OhmRu-content}) that in this terminology, content algebras must be McCoy, and asks \cite[Question 2.12]{Nas-Aus} whether there is any faithfully flat McCoy algebra that is not a content algebra.

After showing that the question itself has a negative answer (i.e. there \emph{is} a faithfully flat McCoy algebra that isn’t a content algebra), we give a partial positive answer to this question by strengthening the premise and weakening the conclusion.  We conclude our note by showing that for faithfully flat Ohm-Rush algebras over a Noetherian base, the corresponding property is indeed equivalent to the weak content and semicontent algebra conditions.

\begin{example}
Let $k$ be a field and let $x,y$ be indeterminates over $k$.  Let $R:=k[x]$ and $S := k[x,y]/ (y^2 - x^3) = R[y] / (y^2 - x^3)$.  Note that $S$ is free as an $R$-module on the basis $\{1,y\}$.  Hence the algebra $R \ra S$ is Ohm-Rush and faithfully flat.  Moreover, since $S$ is an integral domain, it has no nonzero zero-divisors, whence the McCoy property is vacuously satisfied.  But $R \ra S$ is not a content algebra; indeed, it is not even a \emph{weak} content algebra.  To see this, simply note that 
$c(y)^2 = R$, but $c(y^2) = x^3R$, whose radical is the proper ideal $xR$.
\end{example}

However, Nasehpour did have a reason to think the McCoy property was intimately linked to the content algebra property.  As noted above, he proved that for certain monoid algebras $R[M]$ over $R$, these two properties (along with the weak content algebra property) are equivalent.  However, as we saw, this equivalence follows from a properties of the monoid $M$ (namely that it is cancellative and torsion-free), and hence passes to the base-changed residue algebra $R/I \ra R[M] / I R[M] = (R/I)[M]$. Hence, we consider the following

\begin{defn*}
Let $R \ra S$ be an Ohm-Rush algebra, and let $\cI$ be a set of ideals of $R$.  We say the algebra $R \ra S$ is \emph{residually McCoy for $\cI$} if for any ideal $I \in \cI$, the algebra $R/I \ra S/IS$ is McCoy. We say the algebra is \emph{residually McCoy} if it is residually McCoy for the set of all ideals.
\end{defn*}

Recall \cite[Proposition 3.6]{nmeSh-OR2} that if $R \ra S$ is a content algebra (resp. a weak content algebra), so is $R/I \ra S/IS$ for any ideal $I$ of $R$.  Hence content algebras are not merely McCoy, but residually so.  As a partial converse, we have the following.

\begin{thm}\label{thm:wcarmc}
Let $R \ra S$ be an Ohm-Rush algebra.  Then the following are equivalent: \begin{enumerate}
    \item\label{it:wca} $R \ra S$ is a weak content algebra.
    \item\label{it:Mcprime} $R \ra S$ is residually McCoy for prime ideals.
    \item\label{it:Mcrad} $R \ra S$ is residually McCoy for radical ideals.
\end{enumerate}
Hence, a residually McCoy algebra is always a weak content algebra.
\end{thm}

\begin{proof}
We give a circular proof.  The implication (\ref{it:Mcrad}) $\implies$ (\ref{it:Mcprime}) is trivial.

For the implication (\ref{it:Mcprime}) $\implies$ (\ref{it:wca}), let $\p \in \Spec R$.  Let $f, g \in S$ with $fg \in \p S$ and $g \notin \p S$.  Then since $R/\p \ra S/\p S$ is McCoy, there is some $r\in R \setminus \p$ with $\bar{r}\bar{c}(\bar{f}) = \bar{0}$, where $\bar{r} := r+\p \in R/\p$, $\bar{c} := c_{S/\p S, R/\p}$, and $\bar{f} := f+ \p S \in S/\p S$.  By \cite[Remark 2.3(d)]{OhmRu-content}, we have $\bar{c}(\bar{f}) = c_{SR}(f) + \p / \p \subseteq R/\p$.  It follows that $rc(f) \subseteq \p$.  Since $r\notin \p$ and $\p$ is prime, $c(f) \subseteq \p$.  But then by the Ohm-Rush property, $f \in \p S$.  Thus, either $\p S = S$ or $\p S \in \Spec S$, whence $R \ra S$ is a weak content algebra.

For the implication (\ref{it:wca}) $\implies$ (\ref{it:Mcrad}), let $I$ be a radical ideal.  Then by \cite[Proposition 3.6(1)]{nmeSh-OR2}, $R/I \ra S/IS$ is a weak content algebra, whence by \cite[Proposition 2.8]{Nas-zdsemi}, $S/IS$ is a McCoy algebra over $R/I$.
\end{proof}

Hence, any counterexample (a faithfully flat residually McCoy algebra that was not a content algebra) would yield a counterexample to the longstanding open question of whether a faithfully flat weak content algebra must be a content algebra.

To summarize, we have the following diagram of algebra types:
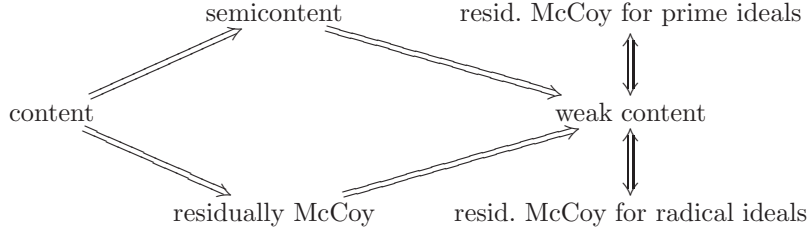
\[\xymatrix{
& \text{semicontent} \ar@{=>}[dr] & \text{resid.\ McCoy for prime ideals} \ar@{<=>}[d] \\
\text{content} \ar@{=>}[ur] \ar@{=>}[dr] &
& \text{weak content} \ar@{<=>}[d] \\
& \text{residually McCoy} \ar@{=>}[ur] & \text{resid.\ McCoy for radical\ ideals}
}\]

Next, we tackle the Noetherian case, or rather the case of a Noetherian base ring.  Recall \cite[Corollary 4.3]{nmeSh-OR2} that a faithfully flat weak content algebra over a Noetherian ring is the same as a semicontent algebra.  So it makes sense to ask whether faithfully flat residually McCoy algebras coincide with the above algebra types.  It turns out that they do.  However, we are able to expand our context a bit beyond Noetherian rings for this.

\begin{defn*}
Recall that a ring $R$ has \emph{property (A)} (also called a \emph{McCoy ring} in some sources) if whenever $J$ is a finitely generated ideal consisting of zero-divisors, there is a nonzero $r\in R$ with $J \subseteq \ann(r)$.  We say that $R$ has the \emph{fidel (A)} property if for any ideal $I$, $R/I$ has property (A).
\end{defn*}
Familiar classes of examples of fidel (A)-rings include: Noetherian rings, B\'ezout rings, zero-dimensional rings, and one-dimensional domains.  See  \cite[Section 4]{AnCh-annmod} for a development of this concept and these facts.  
We note additionally that if $W \subset R$ is a multiplicative set without zero-divisors and $R$ has property (A), then $W^{-1}R$ has property (A) as well, since $r/1\neq 0$ in $W^{-1}R$.

We start with some lemmas for reducing the problem to a more amenable base ring.

\begin{lemma}\label{lem:Mclocalnzd}
Let $R \ra S$ be a flat Ohm-Rush algebra.  Let $W \subseteq R$ be a multiplicative set consisting of regular elements.  Then $R \ra S$ is McCoy if and only if $W^{-1}R \ra W^{-1}S$ is McCoy.
\end{lemma}

\begin{proof}
First suppose $R \ra S$ is McCoy, and let $W$ be a multiplicative set in $R$ containing no zero-divisors.  Let $f/v$, $g/w \in W^{-1}S$ such that $(f/v)(g/w) = 0$, but $g/w\neq 0$.  Then there is some $x\in W$ with $xfg=0$, and we also have $g \neq 0$.  By the McCoy property for $R \ra S$, there is some nonzero $r\in R$ with $rc(xf)=0$.  But by flatness of $S$ over $R$ and \cite[Theorem 1.6]{OhmRu-content}, we then have $rxc(f)=0$, whence by \cite[Theorem 3.1]{OhmRu-content}, we have $(r/1)c_{W^{-1}S, W^{-1}R}(f/v) = 0$.  Moreover, $r/1 \neq 0$ in $W^{-1}R$ since $W$ lacks zero-divisors, finishing the proof that $W^{-1}R \ra W^{-1}S$ is McCoy.

Conversely, suppose $W^{-1}R \ra W^{-1}S$ is McCoy.  Let $f,g \in S$ with $fg=0$ and $g\neq 0$. Then since $R \ra S$ is flat, no element of $W$ can be a zero-divisor in $S$, which means that $g/1 \neq 0$ in $W^{-1}S$.  It follows from the McCoy property in $W^{-1}R \ra W^{-1}S$ that there is some nonzero $r/w \in W^{-1}R$ with $(r/w)c_{W^{-1}S, W^{-1}R}(f/1) =0$.  But then using \cite[Theorem 3.1]{OhmRu-content} again, we have $W^{-1}(rc(f)) = 0$, and since $W$ has no zero-divisors in $R$, we conclude that $rc(f)=0$.
\end{proof}

Next, we show that among flat Ohm-Rush algebras, the property of being McCoy globalizes.

\begin{lemma}\label{lem:Mcglobal}
Let $R \ra S$ be a flat Ohm-Rush algebra such that $R_\m \ra S_\m$ is a McCoy algebra for all maximal ideals $\m$ of $R$, where `$S_\m$' means the localization of $S$ at the multiplicative set $R \setminus \m$.  Then $R \ra S$ is McCoy.
\end{lemma}

\begin{proof}
Let $f,g\in S$ with $fg=0$ and $g\neq 0$.  Then since $R \ra S$ is Ohm-Rush, it follows that $c(g)\neq 0$, whence there is some maximal ideal $\m$ of $R$ with $c(g)_\m \neq 0$. Then by \cite[Theorem 3.1]{OhmRu-content}, we have $c_{R_\m, S_\m}(g/1) = c(g)_\m \neq 0$, whence $g/1 \neq 0$ in $S_\m$.  Since $R_\m \ra S_\m$ is McCoy, it follows that there is some nonzero $r/w \in R_\m$ with $(r c(f))_\m = (r/w) c_{R_\m, S_\m}(f/1) = 0$.  Since $rc(f)$ is a finitely generated ideal \cite[p.\ 51]{OhmRu-content}, there is then some $v \in R \setminus \m$ with $vrc(f)=0$.  But $r/w \neq 0$ in $R_\m$ implies that $vr \neq 0$, completing the proof that $R \ra S$ is McCoy.
\end{proof}

Now we are ready for the result in the locally fidel (A) case. 

\begin{thm}\label{thm:McCoySCAfidel}
Let $R \ra S$ be a semicontent algebra, and assume that for any prime ideal $\p$ of $R$, $R_\p$ is a fidel (A)-ring (e.g. if $R$ is a locally Noetherian, arithmetical, or zero-dimensional ring, or a 1-dimensional domain).  Then $R \ra S$ is residually McCoy.
\end{thm}

\begin{proof}
By \cite[Proposition 3.6]{nmeSh-OR2}, $R/I \ra S/IS$ is a semicontent algebra; hence it is enough to prove that $R \ra S$ is McCoy.  After the above reduction, let $W$ be the set of regular elements of $R$.  By Lemma~\ref{lem:Mclocalnzd} and \cite[Corollary 3.4]{nmeSh-OR} it is enough to show that $W^{-1}R \ra W^{-1}S$ is McCoy.  Thus, we may assume every zero-divisor of $R$ is a unit.  For our final reduction, by Lemma~\ref{lem:Mcglobal} and \cite[Corollary 3.4]{nmeSh-OR}, it is enough to show that $R_\m \ra S_\m$ is McCoy for any maximal ideal $\m$ of $R$. Thus, we may assume that $(R,\m)$ is local and satisfies property (A), and that $\m$ consists of zero-divisors, and we need to show under these circumstances that a semicontent algebra $R \ra S$ is McCoy.

Accordingly, let $f,g \in S$ with $fg=0$ and $g\neq 0$.  If $c(f) = R$, then by the semicontent algebra property we have $c(g) = c(fg) = c(0) = 0$, but by the Ohm-Rush property, $0\neq g \in c(g)S = 0S = 0$, a contradiction.  Thus, $c(f) \subseteq \m$. But then $c(f)$ is a finitely generated ideal consisting of zero-divisors.  Thus by property (A), there is some nonzero $r\in R$ with $rc(f)=0$, finishing the proof that the algebra is McCoy.
\end{proof}

We obtain the following summation for Noetherian base rings as a corollary.

\begin{thm}\label{thm:Noeth}
Let $R \ra S$ be a faithfully flat Ohm-Rush algebra.  Assume $R$ is Noetherian.  Then the following are equivalent: \begin{enumerate}
    \item\label{it:Nwca} $S$ is a weak content $R$-algebra.
    \item\label{it:Nsca} $S$ is a semicontent $R$-algebra.
    \item\label{it:Nrmc} $R \ra S$ is residually McCoy.
    \item\label{it:Nrmcrad} $R\ra S$ is residually McCoy for radical ideals.
    \item\label{it:Nrmcprime} $R \ra S$ is residually McCoy for prime ideals.
\end{enumerate}
\end{thm}

\begin{proof}
The equivalence of (\ref{it:Nwca}) and (\ref{it:Nsca}) is \cite[Corollary 4.3]{nmeSh-OR2}.  The equivalence of (\ref{it:Nwca}), (\ref{it:Nrmcrad}), and (\ref{it:Nrmcprime}) follows from Theorem~\ref{thm:wcarmc}.  Finally, the implication (\ref{it:Nsca}) $\implies$ (\ref{it:Nrmc}) follows from Theorem~\ref{thm:McCoySCAfidel}, and the implication (\ref{it:Nrmc}) $\implies$ (\ref{it:Nrmcrad}) is by definition.
\end{proof}

As an interesting corollary to this along with previous joint work of the author and Shapiro on Ohm-Rush content, we have the following:

\begin{thm}
Let $K$ be a field and let $L/K$ be a regular field extension (i.e. for any extension field $k'$ of $K$, $k' \otimes_K L$ is an integral domain).  Let $R$ be a Noetherian $K$-algebra and let $S := R \otimes_K L$.  Then for any zero-divisor $h\in S$, there is a nonzero element $r\in R$ with $rh=0$.  In particular this holds when $R=K[x_1, \ldots, x_n] / (g_1, \ldots, g_t)$ is a finitely generated $K$-algebra, in which case $S = L[x_1, \ldots, x_n] / (g_1, \ldots, g_t)$.
\end{thm}

\begin{proof}
By \cite[Lemma 4.9]{nmeSh-OR2} and the fact that $S$ is free as an $R$-module, $S$ is a weak content algebra over $R$.  Then the result follows from Theorem~\ref{thm:Noeth}.
\end{proof}

\providecommand{\bysame}{\leavevmode\hbox to3em{\hrulefill}\thinspace}
\providecommand{\MR}{\relax\ifhmode\unskip\space\fi MR }
\providecommand{\MRhref}[2]{%
  \href{http://www.ams.org/mathscinet-getitem?mr=#1}{#2}
}
\providecommand{\href}[2]{#2}

\end{document}